\documentclass{article}
\usepackage{amssymb}
\usepackage{amsmath}
\usepackage{amsthm}
\usepackage{verbatim}
\usepackage{color}
\usepackage{url}
 \usepackage[all]{xy}
\usepackage{fancyhdr}

\newtheorem{thm}{Theorem}

\newtheorem{lem}{Lemma}
\newtheorem{cor}{Corollary}

\newtheorem{defn}{Definition}

\newtheorem{ass}{Assumption}
\newtheorem{rem}{Remark}

\newcommand{\sabs}[1]{\left|#1\right|}
\newcommand{\sparen}[1]{\left(#1\right)}
\newcommand{\norm}[1]{\sabs{\sabs{#1}}}

\numberwithin{equation}{section}

\title{Stability for Time Dependent X-ray Transforms and Applications}
\author{Alden Waters\thanks{Department of Mathematics, University College London, Gower Street, London, WC1E 6BT, United Kingdom}} 
\begin{document}

\maketitle

\begin{abstract}
 We prove a logarithmic stability estimate for the time dependent  X-ray transform on $\mathbb{R}_t^+\times\mathbb{R}^n$. To do so, we extend a known result by Begmatov for the stability of the time dependent X-ray transform in $\mathbb{R}^+_t\times\mathbb{R}^2$. We give some examples of stability and injectivity results in relationship to the Dirichlet-to-Neumann problem. In particular, under the Geometric Control Condtion, we derive inverse logarithmic stability estimates for time dependent conformal factors. 
\end{abstract}
\section{Motivation}

The X-ray transform is a key ingredient in electrical impedance tomography (EIT). However, inverting it by the boundary control method introduced in \cite{Belishev}, \cite{kurylev} is difficult and non-constructive. Typically data in the form of the Dirichlet-to-Neumann map is measured on the exterior of an obstacle, and this is the only information available to scientists who wish to determine the source uniquely. Determining which subsets of the boundary  determine the source of the emitted waves has applications in medical imaging.\\

In this article we consider data measurements on the entirety of the boundary, in a time dependent setting. The results of \cite{waterss} prove that the X-ray transform is recoverable from the boundary measurements for wave equation lower order terms. We obtain a stability estimate for the X-ray transform for certain types of strictly convex domains, so we can reliably recover the sources from boundary data. The time dependent nature of the problem adds new difficulties to the analysis.  \\

The approach of using the X-ray or Radon transforms to recover terms in the wave equation is not new. For time independent coefficients uniqueness was examined in \cite{isakov1}, \cite{rakesh}, \cite{rakeshsymes} and stability in \cite{baoyun}, \cite{DDSF}, \cite{M}, \cite{stefanovstability}, \cite{stefanovstable}. For lower order time dependent coefficients, uniqueness was examined in \cite{stefanov1}, as well as in \cite{sjo}, \cite{s1}, and stability in \cite{s2}. Using the boundary control method under genric hypotheses, stability was also examined in \cite{LO}.  The constructive methodology should be useful in applications to medical imaging. The time independent X-ray transform has already been shown to be useful in this area.\\

The outline of this paper is as follows. Sections 2 and 3 introduce the relevant notation and prove a stability estimate for the time-dependent geodesic X-ray transform in the case of a strictly convex body in $\mathbb{R}^n$. In Section 4 we extend this stability estimate to the case when the manifold is simple, i.e.~manifolds in which the exponential map is a global diffeomorphism. Following \cite{DDSF} and \cite{waterss} we then derive stability estimates for the X-ray transform of time dependent potentials and time dependent conformal factors. \\

Finally, we show that a combination of the main theorems yields the first inverse logarithmic stability estimate for the time dependent conformal factors. This result says the high-frequencies are generically unstable. The assumptions on the conformal factors relate to the geometric control condition of Bardos-Lebeau-Rauch \cite{blr}. The same assumptions have been used in \cite{eskin2}, \cite{eskin1}  to show that certain lower order terms are uniquely determined. For the leading terms the only uniqueness result is \cite{esn}. However, there the use of analyticity in the time variable is essential, since the arguments are predicated on a unique continuation argument from \cite{tataru}. It is unclear from the results in this paper if uniqueness is available without the use of analyticity assumptions. 

\section{Introduction to notational conventions}
The Einstein summation convention is used throughout this article. We write $f \sim g$ if there exists a constant $C>0$ such that $C^{-1}f \leq g \leq Cf$.\\

Let $\mathcal{M}$ be a smooth manifold with boundary and denote by $(\frac{\partial}{\partial x^1}, \dots ,\frac{\partial}{\partial x^n})$ a basis of tangent vector fields. The inner product and norm on the tangent space $T_x\mathcal{M}$ are denoted by 
$$g(X,Y)=\langle X,Y\rangle_g=g_{jk}\alpha_j\beta_k, \qquad |X|_g=\langle X,X \rangle_g^{\frac{1}{2}}, $$
where $X=\alpha_i\frac{\partial}{\partial x_i}$, $Y=\beta_i\frac{\partial}{\partial x_i}$.
The gradient of $f \in C^1(\mathcal{M})$ is defined as the vector field $\nabla_g f$ such that 
\begin{align*}
X(f)=\langle \nabla_g f,X\rangle_g
\end{align*}
for all vector fields $X$ on $\mathcal{M}$. In local coordinates, we can write
\begin{align*}
\nabla_g f=g^{ij}\frac{\partial f}{\partial x_i}\frac{\partial}{\partial x_j}.
\end{align*}
The metric tensor induces a Riemannian volume form, which we denote by 
\begin{align*}
d_gV=(\det g)^{\frac{1}{2}}dx_1\wedge . . . \wedge dx_n.
\end{align*}
The space $L^2(\mathcal{M})$ is the completion of $C^{\infty}(\mathcal{M})$ with respect to the inner product
\begin{align*}
\langle f_1,f_2\rangle=\int\limits_{\mathcal{M}}f_1(x)f_2(x)\,d_gV\ , \qquad f_1,f_2\in \mathcal{C}^{\infty}(\mathcal{M}).
\end{align*}
Sobolev spaces are defined analogously to the Euclidean Sobolev spaces, with the norm 
\begin{align*}
\norm{f}^2_{H^1(\mathcal{M})}=\norm{f}_{L^2(\mathcal{M})}^2+\norm{\nabla f}_{L^2(\mathcal{M})}^2.
\end{align*}
We now introduce the definition of the X-ray transform on $\mathcal{M}$. For $x\in \mathcal{M}$ and $\omega\in T_x\mathcal{M}$ we write $\gamma_{x,\omega}$ for the unique geodesic with initial conditions
\begin{align*}
\gamma_{x,\omega}(0)=x, \qquad \dot{\gamma}_{x,\omega}(0)=\omega.
\end{align*}
We let 
\begin{align}
\mathcal{SM}=\{(x,\omega)\in T\mathcal{M};\quad |\omega|_g=1\}
\end{align}
denote the sphere bundle of $\mathcal{M}$. The submanifold of inner vectors of $\mathcal{SM}$ is given by
\begin{align*}
\partial_{+}\mathcal{SM}=\{(x,\omega)\in \mathcal{SM},\quad x\in \partial\mathcal{M},\quad +\langle\omega,\nu(x)\rangle<0\}.
\end{align*}
Let $\tau(x,\omega)$ be the length of the geodesic segment with initial conditions $(x,\omega)\in\partial_+\mathcal{SM}$. We consider the inward pointing vectors of $\mathcal{SM}$ only.\\ 

The construction of the Gaussian beam Ansatz requires us to make a further assumption on $\mathcal{M}$:

\begin{ass} \label{geoass} There is a $T>0$ such that the following holds: For all $(x,\omega)\in\partial_+\mathcal{SM}$ there is a $\tau(x,\omega)\leq T$ such that $\gamma_{x,\omega}(t)$ is in the interior of $\mathcal{M}$ for $0<t<\tau(x,\omega)$, and intersects the boundary $\partial\mathcal{M}$ transversally when $t=\tau(x,\omega)$.
\end{ass}
This hypothesis is the weakest assumption on the geometry for the Gaussian beam Ansatz to be well-defined in our setting. It is related to the geometric control condition of \cite{blr}. This was also observed in \cite{bao}. 

We now define the X-ray transform of a time dependent, sufficiently smooth function $f(t,x)$ as follows:
\begin{align}\label{Xray}
I_{x,\omega} f=\int\limits_0^{\tau(x,\omega)}f(s,\gamma_{x,\omega}(s))\,ds.
\end{align}
The right hand side of (\ref{Xray}) is a smooth function on the space $\partial_{+}\mathcal{SM}$, because $\tau(x,\omega)$ is a smooth function on $\partial_+\mathcal{SM}$.

\section{Stability Estimate in $\mathbb{R}^n$}
In this section we prove a stability estimate in $\mathbb{R}_t^+\times \mathbb{R}^n$, by extending the result of \cite{bg} in $\mathbb{R}_t^+\times\mathbb{R}^2$.  We prove the following theorem:
\begin{thm}\label{theoremone}
Let $\mathcal{U}\subset\mathbb{R}^n$ be bounded and strictly convex. Then there exists $\epsilon_0 \in (0,1)$ such that for all $f\in C^{\infty}([0,T]\times\mathcal{U})$ with 
\begin{align*}
\norm{If(t,x)}_{C^0(\partial_+\mathcal{S}\mathcal{U})}<\epsilon_0
\end{align*}
we have
\begin{align*}
\norm{f(t,x)}_{L^2([0,T]\times\mathcal{U})}\leq C\sparen{\log \norm{If(t,x)}_{C^0(\partial_+\mathcal{S}\mathcal{U}^+)}^{-1}}^{-1}.
\end{align*}
Here the constant $C$ depends only on $T$, diam$(\mathcal{U})$ and the $C^{n+3}$ norm of $f$. 
\end{thm}

Theorem \ref{theoremone} will be a consequence of the following lemma. 
\begin{lem}\label{lemmamain}
Let $\mathcal{U}\subset\mathbb{R}^n$ be bounded and strictly convex. Then for every $R>1$ and all $f\in C_c^{\infty}([0,T]\times\mathbb{R}^n)$:
\begin{align*}
&\norm{f(t,x)}_{C^0([0,T]\times\mathcal{U})}\leq C_1R^{n+1}\norm{If(t,x)}_{C^{0}(\partial_+\mathcal{S}\mathcal{U})}+\\& \qquad \qquad \qquad \qquad  C_2\exp\sparen{\textstyle{\frac{1}{3}}R}R^{\frac{3n+2}{3}}\norm{If(t,x)}_{C^{0}(\partial_+\mathcal{S}\mathcal{U})}^{\frac{2}{3}}+C_3R^{n+1-a}.
\end{align*}
Here $C_1,C_2,C_3$ are constants that depend on $T$, on diam$(\mathcal{U})$ and $a>n+1$. $C_3$ also depends on the $C^{a+1}([0,T]\times\mathcal{U})$ norm and the size of the support of $f$.
\end{lem}
\begin{proof}[Proof of Theorem \ref{theoremone}]
We set $\delta=\norm{If(t,x)}_{C^{0}(\partial_+\mathcal{SU})}$ and extend $f$ smoothly to $\mathbb{R}^n$. From Lemma \ref{lemmamain} we have
\begin{align}\label{rbound2}
&\norm{f(t,x)}_{C^0([0,T]\times\mathcal{U})}\leq  (C_1+C_2)\exp\sparen{\textstyle{\frac{1}{3}}R}R^{n+1}\delta^{\frac{2}{3}}+C_3R^{n+1-a}.
\end{align}
We would like the first term to satisfy the inequality
\begin{align*}
R^{n+1}\exp\sparen{\textstyle{\frac{R}{3}}}\leq\delta^{\frac{\epsilon}{2}},
\end{align*}
where $\epsilon\in (0,1)$. Taking the logarithm of both sides, this implies
\begin{align*}
(n+1)\log R+\frac{R}{3}\leq \log(\delta^{\frac{\epsilon}{2}-1}).
\end{align*}
Because $\log R\leq R$, this last condition is certainly fufilled if 
\begin{align*}
(n+2)R\leq \log(\delta^{\frac{\epsilon}{2}-1}).
\end{align*}
We would like to sandwich the first term to give $R$ a lower bound
\begin{align*}
R^{n+1}\exp\sparen{\textstyle{\frac{1}{3}}R}\delta\geq \delta^{\epsilon}.
\end{align*}
This is possible when $\frac{R}{3}\geq \log(\delta^{\epsilon-1})$. The result is that $R$ lies in the interval
\begin{align}\label{rsandwich}
\log(\delta^{\epsilon-1})\leq \frac{R}{3}\leq \frac{(\log(\delta^{\frac{\epsilon}{2}}))^{-1}}{n+2},
\end{align}
which is possible provided $\delta$ is sufficiently small. Combining the inequality (\ref{rbound2}) for $a=n+2$ with (\ref{rsandwich}), we obtain 
\begin{align*}
&\norm{f(t,x)}_{C^0([0,T]\times\mathcal{U})}\leq C\sparen{(\log(\delta^{\epsilon-1}))^{-1}+\delta^{\frac{2}{3}+\frac{\epsilon}{2}}}.
\end{align*} 
Notice that the logarithmic term dominates because of the smallness condition on the norm of the ray transform. We conclude the desired result.
\end{proof}

We proceed with the proof of Lemma \ref{lemmamain}, starting with the X-ray transform as defined by \eqref{Xray}. In $\mathbb{R}^n$ the Fourier transform in the $x=(x_1,x_2,. .,x_n)$ variables is related to the definition of the X-ray transform as follows: 
\begin{align*}
\mathcal{F}_{x\rightarrow\xi}(If)(\xi, \omega)=\int\limits_{\mathbb{R}^n}\int\limits_{\mathbb{R}}f(s,x+s\omega)\exp(-ix\cdot\xi)\,ds\,dx\ .
\end{align*}
Here we have that $(x,\omega)\in\partial_+\mathcal{SU}$.  Change of coordinates with $\tilde{x}=x+s\omega$ leads to the following: 
\begin{align*}
\mathcal{F}_{x\rightarrow\xi}(If)(\xi, \omega)=\int\limits_{\mathbb{R}^n}\int\limits_{\mathbb{R}}\exp(-i\tilde{x}\cdot\xi)\exp(-i(-\omega\cdot\xi)s)f(s,\tilde{x})\,ds\,d\tilde{x}\ .
\end{align*}
The right hand side is the Fourier transform in all the variables evaluated at $(-\omega\cdot\xi,\xi)$:
\begin{align*}
\mathcal{F}_{x\rightarrow\xi}(If(\xi, \omega))=\mathcal{F}_{(t,x)\rightarrow(\tau,\xi)}f(-\omega\cdot\xi,\xi) .
\end{align*}
For simplicity we use the following abbreviation for the rest of this paper
\begin{align*}
\hat{f}(\tau,\xi) = \mathcal{F}_{(t,x)\rightarrow(\tau,\xi)}f(\tau,\xi).
\end{align*}
We would like a bound on $|\hat{f}(\tau,\xi)|$ in terms of $\hat{f}(-\omega\cdot\xi,\xi)$. The essential idea is a modification of the arguments from \cite{bg} in $\mathbb{R}_t^+\times\mathbb{R}^2$. Similar to Salazar \cite{s1}, we divide the phase space into two regions,
\begin{align*}
\{(\tau,\xi): |\tau|\leq |\xi| \} \quad \textrm{and} \quad \{(\tau,\xi): |\tau|> |\xi|\}.
\end{align*}
We begin with bounds for $|\hat{f}(\tau,\xi)|$ on the first set, which is the easier part. 
\begin{lem}\label{one}
In $\{(\tau,\xi): |\tau|\leq|\xi| \}$ we have the bound
\begin{align*}
|\hat{f}(\tau,\xi)|=|\hat{f}(-\omega\cdot\xi,\xi)|
\end{align*}
for some $\omega$, provided that $\xi$ corresponds to $x$ with $(x,\omega)\in\partial_+\mathcal{S}\mathcal{U}$. 
\end{lem} 
\begin{proof}[Proof of Lemma \ref{one}]
Since $|\tau|\leq|\xi|$, there exists an inward pointing tangent vector $\omega$ with $(x,\omega) \in \partial_+\mathcal{S}\mathcal{U}$ such that $\tau=\pm\omega\cdot\xi$. The argument of \cite{bg} shows that $|\hat{f}(-\omega\cdot\xi,\xi)| = |\hat{f}(\omega\cdot\xi,\xi)|$.
\end{proof}
In the second region we have the following estimate: 
\begin{lem}\label{two}
In $\{(\tau,\xi): |\tau|> |\xi|\}$ we have the bound
\begin{align*}
 |\hat{f}(\tau,\xi)|\leq C\frac{\exp(\frac{|\tau|}{3})}{|\tau|^{\frac{1}{3}}}\norm{If}_{C^0(\partial_+\mathcal{SU})}^{\frac{2}{3}}.
\end{align*} 
\end{lem}
For the proof, we would like to use an analytic extension argument for the Fourier transform. In order to do so, we recall a Paley-Wiener theorem.
\begin{thm}\label{rudin}[\cite{rudin}, Theorem 19.3]
Let $p: \mathbb{C}\rightarrow \mathbb{C}$ be an entire function such that there exist positive constants $C$ and $A$ with 
\begin{align*}
|p(z)|\leq C\exp(A|z|)\ . 
\end{align*}
Further assume that the restriction to the real axis satisfies $p\in L^2(\mathbb{R})$. Then there exists $h \in L^2(-A,A)$ with 
\begin{align*}
p(z)=\int\limits_{-A}^Ah(w)\exp(-iwz)\,dw .
\end{align*}
\end{thm}
In order to compute bounds on the extension, we also need a lemma on the harmonic measure as in \cite{bg}:
\begin{lem}\label{bg}[Begmatov's Lemma]
In the complex plane $\mathbb{C}$, for some fixed $a>0$ we consider the strip 
\begin{align*}
S=\{z=z_1+iz_2: z_1\in\mathbb{R}, |z_2|<a\}
\end{align*}
and the rays 
\begin{align*}
r_1=\{z: -\infty<z_1\leq -a, z_2=0\}, r_2=\{z: a\leq z_1<\infty, z_2=0\} .
\end{align*}
We set $G=S \setminus \{r_1\cup r_2\}$, the strip $S$ with two rays cut out along the real axis. Let $\mu(z,E,G)$ be the harmonic measure of a set $E$ with respect to $G$. We then have
\begin{align*}
\mu(z,E,G)\sim 1 \ .
\end{align*}
\end{lem}

\begin{proof}[Proof of Lemma \ref{two}]
We first consider $\hat{f}(\tau,\xi)$ for $\tau$ and $\xi'=(\xi_2,. . ,\xi_n)$ fixed. The function $\hat{f}(\tau,\xi_1+i\tilde{\xi}_1,\xi')=\hat{f}_E(\tau,\xi)$ satisfies the hypotheses of Theorem \ref{rudin}. Set $w=x_1$ and $h(w)$ the Fourier transform in all variables except $x_1$. We know $h(w)=\mathcal{F}_{(t,x')\rightarrow(\tau,\xi')}f(\tau,x_1,\xi')$, since $\hat{f}_E(\tau,\xi)$ agrees with $\hat{f}(\tau,\xi)$ on the real axis. Furthermore, $h(w)\in L^2(-A,A)$ for some $A\geq $diam$(\mathcal{U})$. 
 We wish to use the Three Lines Lemma on the complex extension, so it is important to know we can compare the harmonic measure on the boundary to the standard arclength measure as per Begmatov's Lemma \ref{bg}. From the Three Lines Lemma we have as in \cite{bg}: 
\begin{align*}
|\hat{f}(\tau,\xi)|\leq m^{\frac{2}{3}}M^{\frac{1}{3}},
\end{align*}
where
\begin{align*}
m=\sup\limits_{p_{1},p_{2}}|\hat{f}_E(\tau,\xi)|, \qquad M=\sup\limits_{l_{1},l_{2}}|\hat{f}_E(\tau,\xi)|.
\end{align*}
Here the lines $l_1$ and $l_2$ are defined as
\begin{align*}
l_{1}&=\{(\tau, \xi_1+i\tilde{\xi}_1, \xi'): \xi_1\in \mathbb{R}, \tilde{\xi}_1=|\tau| \} \quad \mathrm{and} \\
l_{2}&=\{(\tau, \xi_1+i\tilde{\xi}_1, \xi'): \xi_1\in \mathbb{R}, \tilde{\xi}_1=-|\tau| \},
\end{align*}
and the rays $p_{1}$ and $p_{2}$ are subset of $\{(\tau,\xi): |\tau|\leq |\xi|\}$:
\begin{align*}
p_{1}&=\{(\tau, \xi_1+i\tilde{\xi}_1, \xi'): \xi_1\leq -|\tau|, \tilde{\xi}_1=0 \}
\quad \mathrm{and} \\
p_{2}&=\{(\tau, \xi_1+i\tilde{\xi}_1, \xi'): \xi_1\geq |\tau|, \tilde{\xi}_1=0\}.
\end{align*}
We have already computed upper bounds for $\hat{f}(\tau,\xi)$ in the latter region in Lemma \ref{one}.
Note that for $j=1,2$
\begin{align}\label{above}
\hat{f}_E(\tau,\xi)|_{l_{j}}\sim \int\limits_{\mathbb{R}}\exp(-i(\xi_1+i\tilde{\xi}_1)x_1)\mathcal{F}_{(t,x')\rightarrow(\tau,\xi')}f(\tau,x_1,\xi')\,dx_1.
\end{align}
 We can make this comparison, because the harmonic measure when restricted to the $\xi_1+i0$ axis and the Lebesgue measure are absolutely continuous with respect to each other, again by Begmatov's Lemma. We apply Theorem \ref{rudin} as before to \eqref{above} to obtain 
\begin{align*}
\sup\limits_{l_j}|\hat{f}_E(\tau,\xi)|\leq C\frac{\exp(|\tau|)}{|\tau|}
\end{align*}
for $j=1,2$. Here the constant $C$ depends on the $L^2(-A,A)$ norm of $h$, with $h$ as defined above. Again from the Three Lines Lemma, this results in the following estimate
\begin{align*}
|\hat{f}_E(\tau,\xi)|\leq C\frac{\exp\sparen{\frac{|\tau|}{3}}}{|\tau|^{\frac{1}{3}}}\norm{If}_{C^0(\partial_+\mathcal{SU})}^{\frac{2}{3}}.
\end{align*}
This estimate holds true for all hyperplanes $\tau\neq 0$. When $\tau=0$, $\omega$ in the definition of $If$ needs to be zero and no longer lies in $\partial_+\mathcal{S}\mathcal{U}$. Thus we can exclude the hyperplane $\tau=0$.
\end{proof}

\begin{proof}[Proof of Theorem 1]
From Lemma \ref{one} we have  
\begin{align*}
|\hat{f}(\tau,\xi)|=|\hat{f}(-\xi\cdot\omega,\xi)|
\end{align*}
in the region $\{(\tau,\xi): |\tau|\leq |\xi| \}$, where $(x,\omega)\in\partial_+\mathcal{S}\mathcal{U}$. From Lemma \ref{two} in the region  $\{(\tau,\xi): |\tau|\geq |\xi|\}$ we have
\begin{align*}
|\hat{f}(\tau,\xi)|\leq C|\tau|^{-\frac{1}{3}}\exp\sparen{\frac{|\tau|}{3}}\norm{If}_{C^0(\partial_+\mathcal{SU})}^{\frac{2}{3}}.
\end{align*}
We now bound $f$ in terms of the X-ray transform and known quantities involving $R$. Exactly as in \cite{bg}, in the Fourier inversion formula 
\begin{align*}
f(t,x)=\int\limits_{\mathbb{R}^n}\int\limits_{\mathbb{R}}\hat{f}(\tau,\xi)\exp(ix\cdot\xi+i\tau t)\,d\tau\,d\xi, \quad f\in C_c^{\infty}(\mathbb{R}_t^+\times\mathbb{R}^n),
\end{align*}
we split the domain of integration into $B_{R}= \{(\tau,\xi): |\tau|^2+|\xi|^2<R^2\}$ and $B_{R}^c$, assuming $R>1$. In $B_{R}^c$ we can estimate
\begin{align*}
\iint\limits_{B_{R}^c}|\hat{f}(\tau,\xi)|\,d\tau \,d\xi\leq \iint\limits_{B_{R}^c}\frac{C}{\sparen{|\tau|^2+|\xi|^2}^{\frac{a}{2}}}\,d\tau \,d\xi=C(n,a) R^{n+1-a},
\end{align*}
provided $a>n+1$. Here we have used that $\hat{f}(\tau,\xi)$ decays faster than any polynomial, because the Fourier transform preserves Schwartz functions. The final equality is a well-known $n+1$-dimensional radial integral. 

The interior of $B_R$ is split into two regions. In the first region, Lemma \ref{one} assures 
\begin{align}
\iint\limits_{B_{R}\cap \{(\tau,\xi): |\tau|<|\xi|\}}|\hat{f}(\tau,\xi)|\,d\tau \,d\xi\leq CR^{n+1}\norm{If}_{C^0(\partial_+\mathcal{SU})}\ .
\end{align}
In the second region we may apply Lemma \ref{two} to conclude
\begin{align*}
\iint\limits_{B_{R}\cap \{(\tau,\xi): |\tau|>|\xi|\}}|f(\tau,\xi)|\,d\tau \,d\xi&\leq C \norm{If}_{C^0(\partial_+\mathcal{SU})}^{\frac{2}{3}}\iint\limits_{B_{R}\cap \{(\tau,\xi): |\tau|\geq |\xi|\}}\frac{\exp(\frac{|\tau|}{3})}{|\tau|^{\frac{1}{3}}}\,d\tau \,d\xi 
\\& \qquad \leq C \norm{If}_{C^0(\partial_+\mathcal{SU})}^{\frac{2}{3}}\exp\sparen{\frac{R}{3}}R^{\frac{3n+2}{3}}.
\end{align*}
Here again we have used a well-known radial integral in the last inequality. 
\end{proof}
\section{Extension to simple manifolds}\label{simplesection}
Given $x \in \mathcal{M}$, we define the exponential map $\exp_x: T_x(\mathcal{M})\rightarrow\mathcal{M}$ by
\begin{align}
\exp_x(v)=\gamma_{x,v}(rv), \qquad r=|v|_g.
\end{align}
\begin{defn}
We say that a Riemannian manifold $(\mathcal{M},g)$ is simple if $\partial\mathcal{M}$ is strictly convex with respect to $g$ and for any $x\in\mathcal{M}$ the exponential map $\exp_x: \exp_x^{-1}(\mathcal{M})\rightarrow\mathcal{M}$ is a diffeomorphism. 
\end{defn}
\begin{rem}The case where the map is only surjective is found in older literature. The current defintion of the manifolds identifies them with the strictly geodesically convex subsets of $\mathbb{R}^n$. It is not clear that the methods here are useful in the case when the map is only surjective. \end{rem}

We may use the estimate in Theorem \ref{theoremone} and the coordinate chart given by $\exp_x^{-1}$. More precisely, extend $\mathcal{M}$ to a simple manifold $\mathcal{M}_2$, and consider a simple manifold $\mathcal{M}_1$ such that $\mathcal{M} \Subset \mathcal{M}_1 \Subset \mathcal{M}_2$.
\begin{thm}\label{theoremone}
For every $R>1$ and all $f\in C^{\infty}_c([0,T]\times\mathcal{M}_1)$: 
\begin{align*}
&\norm{f(t,x)}_{L^2([0,T]\times\mathcal{M})}\leq C_1R^{n+1}\norm{If(t,x)}_{C^{0}(\partial_+\mathcal{S}\mathcal{U})}\\ & \qquad \qquad \qquad+ C_2\exp\sparen{\frac{R}{3}}R^{\frac{3n+2}{3}}\norm{If(t,x)}_{C^{0}(\partial_+\mathcal{S}\mathcal{U})}^{\frac{2}{3}}+C_3R^{n+1-a}.
\end{align*}
Here $C_1,C_2,C_3$ are constants that depend on $T$, the $C^{n+3}([0,T]\times\mathcal{U})$ norm of $f$ and the size of the support $\mathcal{M}_1$, on diam$(\mathcal{U})$ and $a>n+1$. 
\end{thm}
\begin{rem} This estimate may look strange because the norm on the right hand side is over $\mathcal{U}$ in $\mathbb{R}^n$, but it will be useful in the next section.  \end{rem}

\section{Hyperbolic inverse boundary value problem}
For this section, let $\mathcal{M}$ be an $n$--dimensional smooth compact Riemannian manifold with smooth boundary. Given positive constants $m_0,M_0$ and $\epsilon$, we define an admissible class of conformal factors $c(t,x)$ as 
\begin{align*}
\mathcal{A}&=\{ m_0\leq c(t,x), \,\, \norm{c(t,x)}_{C^{n+3}([0,T]\times\mathcal{M})}\leq M_0, \, \norm{c-1}_{C^1([0,T]\times\mathcal{M})}<\epsilon,\\ & \qquad\qquad   cg \, \textrm{satisfies Assumption \ref{geoass}} \} \ .
\end{align*}
We consider the solution $u(t,x)$ to the initial-boundary value problem for the wave equation
\begin{align}\label{hyperbolic1}
\Box_{cg}u(t,x)&=0  &\mathrm{on } \quad (0,T)\times \mathcal{M},
\\u(t,x)|_{t=0}=\partial_t u(t,x)|_{t=0}&=0 &  \mathrm{in }\quad \mathcal{M}, \nonumber
\\u(t,x)&=f(t,x) & \mathrm{on } \quad (0,T) \times \partial\mathcal{M}, \nonumber
\end{align} 
where 
\begin{align*}
\Box_{cg}=\partial_t^2-\Delta_{cg}.
\end{align*}
Problem \eqref{hyperbolic1} is well-posed, and we have the following existence and uniqueness result, see e.g.~\cite{lions}:
\begin{lem}\label{wellposed}
Assume that $f\in H_0^1([0,T]\times \partial \mathcal{M})$ and $c\in\mathcal{A}$. Then there exists a unique solution $u \in C([0,T];H^1(\mathcal{M}))\cap C^1([0,T]; L^2(\mathcal{M}))$ to \eqref{hyperbolic1}, and the following
norm bound holds:
\begin{align}\label{normbound}
&\sup_{t\in [0,T]}\sparen{\norm{u(t,x)}_{H^1(\mathcal{M})}+\norm{\partial_tu(t,x)}_{L^2({\mathcal{M})}}}\leq 
C\norm{f(t,x)}_{H^1_0([0,T]\times \partial \mathcal{M})}.
\end{align}
Here the constant $C$ is independent of $f$, but depends on $\norm{c}_{C^{2}([0,T]\times\mathcal{M})}$. 
\end{lem}
Since the problem is well-posed, we may study the inverse problem of recovering the conformal factor from the dynamical Dirichlet-to-Neumann map. We let $\nu=\nu(x)$ be the outer unit normal to $\partial\mathcal{M}$ at $x$ in $\partial\mathcal{M}$, normalized such that 
\begin{align*}
g^{kl}(x)\nu_k(x)\nu_l(x)=1.
\end{align*}
The dynamical Dirichlet-to-Neumann map, $\Lambda_{cg}$, is defined by 
\begin{align*}
\Lambda_{cg}f(t,x)= \nu_k(x)c(t,x) g^{kl}(x)\frac{\partial u}{\partial x_l}(t,x)|_{(0,T)\times \partial \mathcal{M}}.
\end{align*}
A natural norm on the Dirichlet-to-Neumann map is the operator norm between
\begin{align*}
H_0^1(\partial\mathcal{M}\times(0,T))\rightarrow L^2(\partial\mathcal{M}\times(0,T)),
\end{align*}
which we denote by 
\begin{align*}
\norm{\Lambda_{cg}}_{H_0^1\rightarrow L^2}.
\end{align*}
Also c.f.~the discussion in \cite{M}. It follows from Lemma \ref{wellposed} that $\Lambda_{cg}$ is bounded in this norm whenever $c\in \mathcal{A}$. 

For later reference we also require an energy estimate for the inhomogeneous problem 
\begin{align}\label{hyperbolic2}
\Box_{cg}u(t,x)&=F(t,x)  &\mathrm{on } \quad (0,T)\times \mathcal{M},
\\u(t,x)|_{t=0}=\partial_t u(t,x)|_{t=0}&=0 &  \mathrm{in }\quad \mathcal{M}, \nonumber
\\u(t,x)&=0 & \mathrm{on } \quad (0,T) \times \partial\mathcal{M}. \nonumber
\end{align} 
\begin{lem}\label{energyestimate}
Assume that $F\in L^1([0,T];L^2(\mathcal{M}))$ and $c\in\mathcal{A}$. Then there exists a unique solution $u \in C([0,T];H^1_0(\mathcal{M}))\cap C^1([0,T]; L^2(\mathcal{M}))$ to \eqref{hyperbolic2}, and the following
norm bound holds:
\begin{align}\label{energybound}
\hspace*{-0.3cm}\sup_{t\in [0,T]}\sparen{\norm{u(t,x)}_{H^1(\mathcal{M})}+\norm{\partial_tu(t,x)}_{L^2({\mathcal{M})}}}\leq 
C\norm{f(t,x)}_{L^1([0,T];L^2(\mathcal{M}))}.
\end{align}
Here the constant $C$ is independent of $f$, but depends on $\norm{c}_{C^{2}([0,T]\times\mathcal{M})}$. 
\end{lem}

\section{Gaussian beam solutions} \label{GBeams}

We recall the construction of a $0^{th}$ order Gaussian beam in Theorem \ref{gaussians} and resulting corollaries from \cite{waterss}. The constructions in this reference apply in our setting, provided that the manifold satisfies Assumption \ref{geoass} and $c\in\mathcal{A}$. 
\begin{thm}\label{gaussians}
Let $d_g(\cdot,\cdot)$ denote the distance function associated to the Riemannian metric $g$ and $\lambda>1$ an asymptotic parameter. For $c \in \mathcal{A}$ there exist nonzero $a\in H^1([0,T],L^2(\mathcal{M}))$, $\psi\in C^2([0,T]\times \mathcal{M})$ independent of $\lambda$, such that 
\begin{align}\label{gb}
U_{\lambda}(t,x)=\sparen{\frac{\lambda}{\pi}}^{\frac{n}{4}}\exp(i\lambda\psi(t,x))a(t,x)
\end{align}
satisfies
\begin{align*}
\sup\limits_{x\in\mathcal{M}, t\in [0,T]}|\Box_{cg}U_{\lambda}(t,x)|\leq C\lambda^{\frac{n}{4}}.
\end{align*} 
\end{thm}
An iterative construction can be used to construct $N^{th}$ order Gaussian beams which satisfy the wave equation up to an error bounded by $C\lambda^{\frac{n}{4}-N}$. The phase function $\psi$ satisfies an eikonal equation and is such that
\begin{align*}
&\Im\psi(t,\tilde{x}(t))=0,\\&
\Im \psi(t,x)\geq C(t)(x-\tilde{x}(t))^2.
\end{align*}
$\tilde{x}(t)$ is a continuous curve in space-time, defined by the differential equation 
\begin{align}\label{systemomega}
\frac{d\tilde{x}(t)}{dt}=-h_{p}(t,\tilde{x}(t),\omega(t)), \quad \frac{d\omega(t)}{dt}=h_x(t,\tilde{x}(t),\omega(t)), \quad \frac{d\psi(t, \tilde{x}(t))}{dt}=0,
\end{align}
for $(\tilde{x}(t),\omega(t))$ with initial condition $(x_0,\omega_0)\in \partial_+\mathcal{SM}$. The Hamiltonian $h(t,x,p)$ in this equation is given by
\begin{align*}
h(t,x,p)=\sqrt{c(t,x)g^{kl}(x)p_kp_l}.
\end{align*}
The amplitude $a$ of the Gaussian beam $U_{\lambda}$ is obtained from a transport equation. In \cite{waterss}, it was shown that it takes the following form:
\begin{cor}\label{size}
In local coordinates, 
\begin{align*}
a_{0}(t,x)=\sparen{\frac{\det Y(0)}{\det Y(t)}}^{\frac{1}{2}}\sparen{\frac{c(0,x_0)g(x_0)}{c(t,\tilde{x}(t))g(\tilde{x}(t))}}^{\frac{1}{4}}a(0,x)+\mathcal{O}(|x-\tilde{x}(t)|).
\end{align*}
\end{cor}

\begin{cor}\label{cutoff}
For $\epsilon_1 \in (0,1)$, there exists a cutoff function $\chi_{\epsilon_1}\in C^{\infty}([0,T]\times\mathcal{M})$ such that for $\alpha>1$: 
\begin{align*}
&\chi_{\epsilon_1}(t,x)=0 \qquad \mathrm{if} \qquad (t,x)\in \{(s, y): \exists r\in [0,T] \ \text{s.t.}\  |s-r|+d_g(y,\tilde{x}(r))>2^{\frac{1}{n}}\epsilon_1^{\frac{1}{2n\alpha}}\}, \\& 
\chi_{\epsilon_1}(t,x)=1  \qquad \mathrm{if} \qquad (t,x)\in \{(s,y): \exists r\in [0,T]\ \text{s.t.}\  |s-r|+d_g(y,\tilde{x}(r))<\epsilon_1^{\frac{1}{2n\alpha}}\},\end{align*}
and furthermore for $m\in \mathbb{N}$: 
\begin{align*}
\sup\limits_{t,x}|\nabla_g^{m}(\chi_{\epsilon_1}(t,x))|<C\epsilon_1^{-\frac{m}{2\alpha}}
\end{align*}
for a constant $C$ which is independent of $\epsilon_1$ and $n=\mathrm{dim}(\mathcal{M})$.
\end{cor}
Directly from Theorem \ref{gaussians} we observe the following corollary:
\begin{cor} \label{boundarybeams}
Given $\epsilon_1\in (0,1)$,  $(x_0,\omega_0)\in  \partial_+\mathcal{SM}$ and $t_0>0$, we can build a $0^{th}$-order Gaussian beam $U_{\lambda}$ which satisfies:
 \begin{enumerate}
 \item $a_0(t_0,x_0)=1$
 \item $\nabla_x\psi(t_0,x_0)=-\omega_0$ 
 \item $\psi_t(t_0,x_0)=1$
 \item $\mathrm{supp} \ U_{\lambda}(t_0,x)\subset B_{\epsilon_1^{\frac{1}{2n\alpha}}}(t_0,x_0)$.
 \end{enumerate}
\end{cor} 

For the construction of the Ansatz as in \cite{waterss} on the manifold we also have the additional assumption that $cg$ does not have any closed loops, although many of the well posedness estimates are formulated without this assumption. 

Motivated by \cite{DDSF}, Section 6.1, we consider perturbations of the above Gaussian beam. Using the cut-off from Corollary \ref{cutoff}, one constructs localized actual solutions to the wave equation. Assume that $\psi_2$ and $\psi_3$ solve the eikonal equation with metric $g$, respectively $cg$, up to higher-order terms:
\begin{align*}
&|\nabla_g\psi_2|^2_g=g^{jk}\partial_{x_j}\psi_2\partial_{x_k}\psi_2=(\partial_t \psi_2)^2 +\mathcal{O}(|x-x(t)|^5), \\& 
 |\nabla_{cg}\psi_3|^2_{cg}=cg^{jk}\partial_{x_j}\psi_3\partial_{x_k}\psi_3=(\partial_t \psi_3)^2+\mathcal{O}(|x-\tilde{x}(t)|^5). 
\end{align*}
Here $x(t)$ is the curve associated to the Hamiltonian $h$ with conformal factor $c=1$. We similarly consider solutions $a_2, a_3$ of perturbed transport equations:
\begin{align*}
&\partial_t \psi_2 \partial_ta_2-g^{jk}\partial_{x_j}\psi_2\partial_{x_k}a_2+\frac{a_2}{2}(\partial_t^2-\nabla_g)\psi_2=\mathcal{O}(|x-x(t)|^3),\\ 
&\partial_t \psi_3\partial_ta_3-cg^{jk}\partial_{x_j}\psi_3\partial_{x_k}a_3+\frac{a_3}{2}(\partial_t^2-\nabla_{cg})\psi_3=-\phi_0 a_2+\mathcal{O}(|x-\tilde{x}(t)|^3). 
\end{align*}
Here $$\phi_0(t,x)=-\frac{1-c^{-1}}{2i} g^{jk}\partial_{x_j}\psi_2(t,x)\partial_{x_k}\psi_2(t,x) e^{i\lambda (\psi_2(t,x) - \psi_3(t,x))}.$$ As $c \in \mathcal{A}$, $\|1-c^{-1}\|_{C^1} \leq C \epsilon$. The well-posedness of the transport equation implies that in the norm $\norm{A}_{*, \mathcal{M}} =\norm{A}_{H^1(0,T, H^2(\mathcal{M}))} + \norm{A}_{H^3(0,T, L^2(\mathcal{M}))}$,
\begin{align*}
\norm{a_3}_{*, \mathcal{M}} \leq C\epsilon\lambda^2 \norm{a_2}_{*, \mathcal{M}} + C
\end{align*}
and in the tube $\mathcal{T} = \mathrm{supp}\ \chi_{\epsilon_1}$, 
\begin{align*}
\norm{a_3}_{*, \mathcal{T}} \leq C\epsilon\lambda^2 \norm{a_2}_{*, \mathcal{T}} + o_{\epsilon_1}(1).
\end{align*}


We consider a modification of Lemma 6.2 in \cite{DDSF}: 
\begin{lem}\label{perturblemma}
Let $c\in \mathcal{A}$ be such that $\epsilon$ is sufficiently small and $c=1$ near the boundary $\partial\mathcal{M}$. The problem
\begin{align*}
(\partial_t^2-\Delta_{cg})u=0 \qquad \textrm{in} \qquad [0,T]\times\mathcal{M}, \qquad u(0,x)=\partial_tu(0,x)=0,
\end{align*}
has a solution of the form
\begin{align*}
u_2(t,x)=\frac{1}{\lambda}a_2(t,x)\exp(i\lambda\psi_2(t,x))\chi_{\epsilon_1}(t,x)+a_3(t,x)\exp(i\lambda\psi_3(t,x)){\chi}_{\epsilon_1}(t,x)+R_{\lambda}(t,x).
\end{align*}
When $\lambda$ is sufficiently large, the remainder $R_\lambda$ satisfies the estimate \begin{align*}
&\sup_{t\in[0,T]}\sparen{\lambda\norm{R_{\lambda}}_{L^2(\mathcal{M})}+\norm{\nabla_gR_{\lambda}}_{L^2(\mathcal{M})}+\norm{\partial_tR_{\lambda}}_{L^2(\mathcal{M})} } \\& \qquad\leq C\sparen{\epsilon\lambda^2+\frac{1}{\lambda}}\norm{a_2}_* + o_{\epsilon_1}(1) + \mathcal{O}\left( e^{-\lambda \epsilon_1^{1/2n\alpha}}\frac{\epsilon_1^{-1/\alpha}}{\lambda}\right). 
\end{align*}
The constant $C$ depends only on $T$ and $\mathcal{M}$. 
\end{lem} 
For the proof, we use the energy estimate \eqref{energyestimate} for the solution $R_{\lambda}$ of the inhomogeneous equation
\begin{align}
\Box_{cg}R_{\lambda}(t,x)=(\partial_t^2-\Delta_{cg})R_{\lambda}(t,x)=F(t,x),
\end{align}
for a right hand side $F \in L^1([0,T]; L^2(\mathcal{M}))$ and zero initial and boundary conditions.
In our case
\begin{align*}
-F(t,x)&= \frac{1}{\lambda}\exp(i\lambda\psi_2)\Box_{cg}a_2\chi_{\epsilon_1}\\&\qquad+2i\exp(i\lambda\psi_2)\chi_{\epsilon_1}\sparen{-\partial_t \psi_2 \partial_t a_2+c^{{-1}}g^{jk}\partial_{x_j}\psi_2\partial_{x_k}a_2+\frac{a_2}{2}\Delta_{cg}\psi_2 - \frac{a_2}{2}\partial_t^2 \psi_2 }\\&\qquad+
\lambda a_2\exp(i\lambda\psi_2)\chi_{\epsilon_1}\sparen{(\partial_t \psi_2)^2-c^{-1}g^{jk}\partial_{x_j}\psi_2\partial_{x_k}\psi_2}+\exp(i\lambda\psi_3)\Box_{cg}a_3\tilde{\chi}_{\epsilon_1}\\&\qquad+
2i\lambda\exp(i\lambda\psi_3){\chi}_{\epsilon_1}\sparen{-\partial_t \psi_3 \partial_ta_3+cg^{jk}\partial_{x_j}\psi_3\partial_{x_k}a_3+\frac{a_3}{2}\Delta_{cg}\psi_3-\frac{a_3}{2}\partial_t^2\psi_3}\\&\qquad+\lambda^2a_3\exp(i\lambda\psi_3){\chi}_{\epsilon_1}\sparen{(\partial_t \psi_3)^2-cg^{jk}\partial_{x_j}\psi_3\partial_{x_k}\psi_3} + E(t,x).
\end{align*}
Here $E(t,x)$ contains the lower-order terms which involve derivatives of $\chi_{\epsilon_1}$ and is of order $\mathcal{O}\left( e^{-\lambda \epsilon_1^{1/2n\alpha}}\epsilon_1^{-1/\alpha}\right)$. Using the eikonal equation for $\psi_2, \psi_3$ and the transport equation for $a_2$,
\begin{align*}
-F(t,x)&= \frac{1}{\lambda}\exp(i\lambda\psi_2)\Box_{cg}a_2\chi_{\epsilon_1}+\exp(i\lambda\psi_3)\Box_{cg}a_3\chi_{\epsilon_1}\\&\qquad+2i\exp(i\lambda\psi_2)\chi_{\epsilon_1}\sparen{(c^{-1}-1)g^{jk}\partial_{x_j}\psi_2\partial_{x_k}a_2+\frac{a_2}{2}(\Delta_{cg}-\Delta_g)\psi_2+\mathcal{O}(|x-x(t)|^3)}\\&\qquad+
\lambda a_2\exp(i\lambda\psi_2)\chi_{\epsilon_1}\sparen{(1-c^{-1})g^{jk}\partial_{x_j}\psi_2\partial_{x_k}\psi_2 +\mathcal{O}(|x-x(t)|^5)}\\&\qquad+
2i\lambda\exp(i\lambda\psi_3)\chi_{\epsilon_1}\sparen{-\partial_t \psi_3 \partial_ta_3+cg^{jk}\partial_{x_j}\psi_3\partial_{x_k}a_3+\frac{a_3}{2}\Delta_{cg}\psi_3-\frac{a_3}{2}\partial_t^2\psi_3}\\&\qquad+\mathcal{O}(\lambda^2|x-\tilde{x}(t)|^5)+E(t,x).
\end{align*} 
Up to the error terms, denoted by $\tilde{E}$, the transport equation for $a_3$ results in
\begin{align*}
-F(t,x)&=\frac{1}{\lambda}\exp(i\lambda\psi_2)\Box_{cg}a_2\chi_{\epsilon_1}+\exp(i\lambda\psi_3)\Box_{cg}a_3\chi_{\epsilon_1}
\\&\qquad+2i\exp(i\lambda\psi_2)\chi_{\epsilon_1}\sparen{(c^{-1}-1)\langle{\nabla_g\psi_2,\nabla_ga_2\rangle}_g + \frac{a_2}{2}(\Delta_{cg}\psi_2-\Delta_g\psi_2)} + \tilde{E}(t,x)\\& =
\frac{1}{\lambda}\exp(i\lambda\psi_2)k_0+\exp(i\lambda\psi_2)k_1+\exp(i\lambda\psi_3)k_2 + \tilde{E}(t,x).
\end{align*}
As
\begin{align*}
k_j\in L^1([0,T];L^2(\mathcal{M})) ,
\end{align*}
we obtain from the energy estimate \eqref{energyestimate} for the inhomogeneous wave equation, that
\begin{align*}
R_{\lambda}\in C^1([0,T];L^2(\mathcal{M}))\cap C([0,T];H_0^1(\mathcal{M}))
\end{align*}
and, up to a term from $\tilde{E}$,
\begin{align*}
\sup_{t\in[0,T]}\norm{R_{\lambda}}_{L^2(\mathcal{M})}&\leq \frac{C}{\lambda}\int_0^T\left\{\frac{\norm{k_0(s,\cdot)}_{L^2(\mathcal{M})}}{\lambda}+\norm{k_1(s, \cdot)}_{L^2(\mathcal{M})}+\norm{k_2(s, \cdot)}_{L^2(\mathcal{M})}\right\}\ ds\\
&\quad + C \int_0^T\left\{\frac{\norm{\partial_s k_0(s,\cdot)}_{L^2(\mathcal{M})}}{\lambda}+\norm{\partial_sk_1(s, \cdot)}_{L^2(\mathcal{M})}+\norm{\partial_s k_2(s, \cdot)}_{L^2(\mathcal{M})}\right\}\ ds.
\end{align*}
The contribution from $\tilde{E}$ can be seen to be $o_{\epsilon_1}(1) + \mathcal{O}\left( e^{-\lambda \epsilon_1^{1/2n\alpha}}\frac{\epsilon_1^{-1/\alpha}}{\lambda}\right)$.
By the definition of the $k_j$, the right hand side is bounded by
$$\frac{C}{\lambda}\left(\frac{1}{\lambda}\norm{a_2}_*+\epsilon\norm{a_2}_*+\epsilon\lambda^2\norm{a_2}_*+o_{\epsilon_1}(1)\right) \ .$$
Similarly,
\begin{align*}
\norm{F}_{L^2([0,T]\times\mathcal{M})}\leq C\sparen{\frac{1}{\lambda}\norm{a_2}_*+\epsilon\norm{a_2}_*+\epsilon\lambda^2\norm{a_2}_*}+ o_{\epsilon_1}(1) + \mathcal{O}\left( e^{-\lambda \epsilon_1^{1/2n\alpha}}\frac{\epsilon_1^{-1/\alpha}}{\lambda}\right).
\end{align*}
Applying the energy estimate again, we obtain
\begin{align*}
\sup_{t\in[0,T]}\sparen{\norm{\nabla R_{\lambda}}_{L^2(\mathcal{M})}+\norm{\partial_tR_{\lambda}}_{L^2(\mathcal{M})}}\leq C\sparen{\epsilon\lambda^2+\frac{1}{\lambda}}\norm{a_2}_{*}+ o_{\epsilon_1}(1) + \mathcal{O}\left( e^{-\lambda \epsilon_1^{1/2n\alpha}}\frac{\epsilon_1^{-1/\alpha}}{\lambda}\right).
\end{align*}
This finishes the proof of Lemma \ref{perturblemma}.\\

Given  $c\in \mathcal{A}$, we now denote
\begin{align*}
&\rho_0(t,x)=1-c(t,x), \\& \rho_1(t,x)=c^{\frac{n}{2}}(t,x)-1, \qquad \rho_2(t,x)=c^{\frac{n}{2}}(t,x)-1, \\&
\rho(t, x)=\rho_1(t,x)-\rho_2(t, x)=c^{\frac{n}{2}-1}(t,x)(c(t,x)-1).
\end{align*}
From the definitions, we observe that there exists $C>0$ with
\begin{align*}
&\norm{\rho_j}_{C^1([0,T]\times\mathcal{M})}\leq C\norm{\rho_0}_{C([0,T]\times\mathcal{M})}, \qquad j=1,2, \\&
C^{-1}\norm{\rho_0}_{L^2([0,T]\times\mathcal{M})}\leq \norm{\rho}_{L^2([0,T]\times\mathcal{M})}\leq C.
\end{align*}

The following key identity relates the difference of the Dirichlet--Neumann operators to the solutions of the corresponding wave equations in the interior: 
\begin{lem}\label{keyidentity}
Let $T>0$ and $c\in \mathcal{A} \cap C^{\infty}([0,T]\times\mathcal{M})$ be such that $c=1$ near the boundary $\partial\mathcal{M}$. Assume that $u_1$ and $u_2$ solve the following problems in $[0,T]\times\mathcal{M}$ for $f_1,f_2\in H^1_0([0,T]\times\partial\mathcal{M})$: 
\begin{align*}
\partial_t^2-\Delta_g u_1&=0 & \mathrm{in} \ [0,T]\times\mathcal{M}, \\
u_1(0,x)=\partial_tu_1(0,x)&=0 & \mathrm{in} \ \mathcal{M}, \\
u_1&=f_1 & \mathrm{on} \ (0,T)\times\partial\mathcal{M}, 
\end{align*}
and
\begin{align*}
\partial_t^2-\Delta_{cg} u_2&=0 & \mathrm{in} \ [0,T]\times\mathcal{M}, \\
u_2(T,x)=\partial_tu_2(T,x)&=0 & \mathrm{in} \ \mathcal{M}, \\
u_2&=f_2 & \mathrm{on} \ (0,T)\times\partial\mathcal{M}. 
\end{align*}
We then have the following identity:
\begin{align*}
&\int\limits_0^T\int\limits_{\partial\mathcal{M}}(\Lambda_g-\Lambda_{cg})f_1\overline{f_2}\,d\sigma_g^{n-1}\,dt\\&\qquad =
\int\limits_0^T\int\limits_{\mathcal{M}}\rho_1(x)\partial_tu_1\partial_t\overline{u}_2\,d_gV\,dt
-\int\limits_0^T\int\limits_{\mathcal{M}}\rho_2(t,x)\langle \nabla_g u_1(t,x), \nabla_g \overline{u}_2(t,x)\rangle_g \,d_gV\,dt .
\end{align*} 
\end{lem}
See \cite{DDSF}, Lemma 6.1, for a proof. There the conformal factors are time independent, but the proof applies in this time dependent setting.

Substituting the Gaussian beam solutions into the key identity from Lemma \ref{keyidentity}, we conclude: 
\begin{lem}
Let $\epsilon>0$. There exists a constant $C>0$ such that for any $a_1,a_2\in H^1(\mathbb{R}_t^+,H^2(\mathcal{M}))$ satisfying the transport equation to leading order we have for all sufficiently large $\lambda$:
\begin{align*}
&\sabs{\int\limits_0^T\int\limits_{\mathcal{M}}\rho(t,x)a_1\overline{a}_2(t,x)\,d_gV\,dt}\\ &\qquad \qquad \leq C \norm{\rho_0}_{\mathcal{C}(\mathcal{M}\times[0,T])}\sparen{\lambda^{-1}+\epsilon\lambda^3}\norm{a_1}_*\norm{a_2}_*\\& \qquad \qquad \qquad+ C\lambda^3\norm{a_1}_*\norm{a_2}_*\norm{\Lambda_g-\Lambda_{cg}}_{H_0^1\rightarrow L^2}  \\& \qquad \qquad \qquad+\norm{\rho_0}_{\mathcal{C}(\mathcal{M}\times[0,T])}(\norm{a_1}_*+\norm{a_2}_*) \Big(o_{\epsilon_1}(\lambda ) + \mathcal{O}\left( e^{-\lambda \epsilon_1^{1/2n\alpha}}\epsilon_1^{-1/\alpha}\Big)\right).
\end{align*}
\end{lem}
The proof follows the arguments of \cite{DDSF}, Lemma 6.3, using Lemma \ref{perturblemma} instead of Lemma 6.2.


\section{Stability estimate for conformal factors} 
We finally use the results of the previous section to deduce a stability estimate for the conformal factors.
\begin{lem}\label{lemmacon}
We have that
\begin{align*}
\sabs{\int\limits_0^T\rho(t,\tilde{x}(t))\,dt}&\leq C \norm{\rho_0}_{C(\mathcal{M}\times[0,T])}\sparen{\lambda^{-1}+\epsilon\lambda^3}\norm{a_1}_*\norm{a_2}_*+\\&\qquad C\lambda^3\norm{a_1}_*\norm{a_2}_*\norm{\Lambda_g-\Lambda_{cg}}_{H_0^1\rightarrow L^2}+\\ & \qquad\sparen{\frac{2\lambda^{\sigma}\epsilon_1^{-\frac{1}{2\alpha}}}{\sqrt{\lambda}}+4\mathrm{erfc}(-\lambda^{2\sigma})}\norm{\rho}_{C^1((0,T)\times \mathcal{M})}+ \\& \qquad \norm{\rho_0}_{\mathcal{C}(\mathcal{M}\times[0,T])}(\norm{a_1}_*+\norm{a_2}_*) \Big(o_{\epsilon_1}(\lambda ) + \mathcal{O}\left( e^{-\lambda \epsilon_1^{1/2n\alpha}}\epsilon_1^{-1/\alpha}\Big)\right).
\end{align*}
\end{lem}
For the proof, we recall a convergence lemma from \cite{steinreal}: 
\begin{lem}\label{H}
Let $h(t,x)\in C^1((0,T)\times O)$, where $O$ is an open subset of $\mathbb{R}^n$ and $B$ be a symmetric nonsingular matrix such that $\Re{B}\geq 0$. If $\tilde{x}(t)$ is a continuous curve in $O$, we have the following uniform estimate:
\begin{align*}\label{rrep}
&\sabs{\sparen{\frac{\lambda}{\pi}}^{\frac{n}{2}}(\det B)^{\frac{1}{2}}\int\limits_{O}\exp\sparen{\langle-\lambda B(x-\tilde{x}(t)), (x-\tilde{x}(t)\rangle} h(t,x)\chi_{\epsilon_1}(t,x)\,dx-h(t,\tilde{x}(t))} \nonumber\\& 
\qquad \qquad \leq\sparen{\frac{2\lambda^{\sigma}\epsilon_1^{-\frac{1}{2\alpha}}}{\sqrt{\lambda}}+4\mathrm{erfc}(-\lambda^{2\sigma})}\norm{h(t,x)}_{C^1((0,T)\times O)}.
\end{align*}
Here $\chi_{\epsilon_1}(t,x)$ is defined as in Corollary \ref{cutoff}, but with respect to the Euclidean metric. 
\end{lem}
To obtain Lemma \ref{lemmacon}, note that the leading order terms $a_2$ and $a_3$ of the Gaussian beam both take the form as in Corollary \ref{size}, but with $c=1$.  The assertion follows by combining this with Lemma \ref{H}:
\begin{align*}
a_{2}(t,x)=\sparen{\frac{\det Y(0)}{\det Y(t)}}^{\frac{1}{2}}\sparen{\frac{g(x_0)}{g({x}(t))}}^{\frac{1}{4}}a(0,x)+\mathcal{O}(|x-{x}(t)|).
\end{align*}
Here one recalls from Section \ref{GBeams} that
\begin{align*}
\sparen{\frac{\det Y(0)}{\det Y(t)}}^{\frac{1}{2}}\Im\psi_2(t, x)=C(t)>0.
\end{align*} 
We then have the following theorem:
\begin{thm}
Let $\mathcal{M}$ be a simple manifold, and $c\in \mathcal{A}$. Assume that
\begin{align*}
\norm{\Lambda_g-\Lambda_{cg}}_{H_0^1\rightarrow L^2}<\epsilon_0
\end{align*}
for a sufficiently small $\epsilon_0$. Then
\begin{align*}
\norm{1-c(t,x)}_{L^2([0,T]\times\mathcal{M})}\leq C\sparen{\log \norm{\Lambda_g-\Lambda_{cg}}_{H_0^1\rightarrow L^2}^{-1}}^{-1},
\end{align*}
where $C$ depends on the $C^{n+3}$ norm of $c$, as well as on $T$, the metric $g$, the diameter of $\mathcal{M}$ and $\epsilon_0$. 
\end{thm}
We refer to \cite{DDSF}, \cite{waterss} for details. One uses Lemma \ref{lemmacon} and minimizes the right hand side in $\lambda = \lambda(\epsilon)$, for small $\epsilon=\epsilon_1$. The conclusion is obtained by choosing
\begin{align*}
\epsilon = \norm{\Lambda_g-\Lambda_{cg}}_{H_0^1\rightarrow L^2}.
\end{align*}

\subsection*{Acknowledgements}
A.W.\ acknowledges support by ERC Advanced Grant MULTIMOD 26718.

\renewcommand\refname{\large References}
\bibliographystyle{abbrv}
\bibliography{aldenbibpbeams}
\end{document}